\documentclass[a4paper,11pt]{article}
\usepackage{latexsym}
\usepackage{amssymb}
\usepackage{enumerate}
\usepackage{amsfonts}
\usepackage{amsmath}
\usepackage[dvipdfmx]{graphicx} 
\usepackage[paper=a4paper,left=30mm,right=20mm,top=25mm,bottom=30mm]{geometry}
    
\newcommand{\qed}{$\Box$}

\newenvironment{@abssec}[1]{%
    \if@twocolumn

      \section*{#1}%
    \else

      \vspace{.05in}\footnotesize
      \parindent .2in
 {\upshape\bfseries #1. }\ignorespaces
    \fi}

    {\if@twocolumn\else\par\vspace{.1in}\fi}

\newenvironment{keywords}{\begin{@abssec}{\keywordsname}}{\end{@abssec}}

\newenvironment{AMS}{\begin{@abssec}{\AMSname}}{\end{@abssec}}

\newcommand\keywordsname{Key words}
\newcommand\AMSname{AMS subject classifications}
\newcommand\AMname{AMS subject classification}
\newtheorem{theorem}{Theorem}
 \newtheorem{lemma}[theorem]{Lemma}
 \newtheorem{corollary}[theorem]{Corollary}
 \newtheorem{proposition}[theorem]{Proposition}

\def\qed{\vbox{\hrule height0.6pt\hbox{%
  \vrule height1.3ex width0.6pt\hskip0.8ex
  \vrule width0.6pt}\hrule height0.6pt
 }}

\title{Interaction between nonlinear diffusion and mean curvature of a  surface with initial discontinuities\thanks{This research was partially supported by the Grant-in-Aid
for Scientific Research  (C) ($\sharp$ 26K06856)  of
Japan Society for the Promotion of Science. }}

\author{Shigeru Sakaguchi\thanks{Admissions Center, Tohoku University, 
Sendai, 980-8576, Japan (sigersak@tohoku.ac.jp).}}

\date{}
\begin{document}
\maketitle

\begin{abstract}
We consider the Cauchy problems for a class of nonlinear diffusion equations where each initial data is given by a characteristic function of an open set $\Omega$ in the whole Euclidean space.
If the boundary $\partial\Omega$ is of class  $C^2$ in a neighborhood of a point $x$ on it, then we extract the mean curvature of $\partial\Omega$ at $x$ from the initial diffusion at $x$. This  corresponds to the author's previous result  for the two-phase linear diffusion equation. Applications to some overdetermined problems for nonlinear diffusion equations are given.
 \end{abstract}



\pagestyle{plain}
\thispagestyle{plain}


\section{Introduction}
\label{introduction}

In the previous paper \cite{Sa2025},  the author considered the Cauchy problem for  the two-phase linear diffusion equation  in the whole Euclidean space consisting of two media with different constant conductivities, where initial data is given by a characteristic function of one medium $\Omega$. There,  if  the boundary $\partial\Omega$ is of class  $C^2$ in a neighborhood of a point $x$ on it, then the mean curvature of $\partial\Omega$ at $x$ is extracted from the initial diffusion at $x$. The present paper considers a class of nonlinear diffusion equations instead of  the two-phase linear diffusion equation.

To be precise, let $\Omega$ be an open set in $\mathbb R^N$ with $N \ge 2$, and let $p \in \partial\Omega$.  Denote by $B_r(x)$ an open ball in $\mathbb R^N$ with radius $r > 0$  and centered at a point $x \in \mathbb R^N$. Assume that there exists  $\rho > 0$ such that $\partial\Omega\cap B_\rho(p)$ is of class $C^2$ and $\partial\Omega$ separates $B_\rho(p)$ into two connected components. Let $\phi : \mathbb R \to \mathbb R$ satisfy
\begin{equation}
\label{nonlinear diffusion}
\phi \in C^2(\mathbb R),\ \phi(0)=0,\ \mbox{ and }\ 0 < \delta_1 \le \phi^\prime(s) \le \delta_2 \ \mbox{ for }  s \in \mathbb R,
\end{equation}
where $\delta_1, \delta_2$ are positive constants, and 
 let $u=u(x,t)$ be the unique bounded solution of the Cauchy problem for the nonlinear diffusion equation:
\begin{equation}
\label{Cauchy problem}
u_t=\Delta \phi(u)\ \mbox{ in }\ \mathbb R^N\times(0,+\infty)\ \mbox{ and }\ u =\mathcal{X}_{\Omega}\ \mbox{ on }\ \mathbb R^N\times\{0\},
\end{equation}
where $\mathcal X_\Omega$ denotes the characteristic function of the set $\Omega$. This Cauchy problem has been dealt with in the papers \cite{MS2010, MS2012, Sa2013} which consider some  overdetermined problems.
The maximum  principle gives
\begin{equation}
\label{positive values}
0 < u < 1\ \mbox{ in } \mathbb R^N\times (0, + \infty).
\end{equation}

The first purpose of the present paper is to show the following theorem:

\begin{theorem}
\label{th: asymptotic formula in time}  There exists a positive constant $c=c(\delta_1, \delta_2) \le 1$ depending only on $\delta_1, \delta_2$ such that if $\phi \in C^3(\mathbb R)$ satisfies $\max\limits_{-1\le s\le2}\{|\phi^{\prime\prime\prime}(s)|, |\phi^{\prime\prime}(s)|\} \le c$, then 
the following formula holds true for the solution  $u$ of \eqref{Cauchy problem} and every point $x \in \partial\Omega\cap B_{\rho}(p)$:
\begin{equation}
\label{asymptotic formula in time}
\lim_{t \to 0^+}t^{-\frac 12}\left\{u(x,t)-f(0)\right\} = \gamma(\phi) (N-1)H(x)\ \mbox{ and } \ \delta_1f^\prime(0)\le\gamma(\phi)\le \delta_2f^\prime(0),
\end{equation}
where $H(x)$ denotes the mean curvature of $\partial\Omega$ at $x\in\partial\Omega\cap B_\rho(p)$ with respect to the outward normal direction to $\partial\Omega$, $\gamma(\phi)$ is a constant depending only on $\phi$, and $f=f(\xi)$ for $\xi \in \mathbb R$ is the unique solution of
\begin{equation}\nonumber
\begin{cases}
&\left(\phi^\prime(f)f^\prime\right)^\prime+\frac12\xi f^\prime=0 \mbox{ and } f^\prime>0\ \mbox{ in } \mathbb R,\\
 &f(-\infty)=0\ \mbox{ and }\  f(+\infty)=1.
 \end{cases}
 \end{equation}
The convergence in \eqref{asymptotic formula in time} is uniform on $\partial\Omega\cap \overline{B_r(p)}$ for each $0 < r < \rho$.
\end{theorem}
When $\phi(s)\equiv s$, that is, when $\delta_1=\delta_2=1$ and the equation is just the heat equation $u_t=\Delta u$, this formula can be directly obtained with the aid of the Gaussian kernel. Indeed, in \cite[Theorem 4.1, pp.546--548]{E1993}, by using  the explicit representation of temperature,  Evans shows  that initially the level surface of temperature $\frac 12$ moves with normal velocity $(N-1)H$. Moreover, Theorem \ref{th: asymptotic formula in time} covers the case for the heat equation $u_t=\Delta u$. Indeed, when $\phi(s)\equiv s$ together with $\delta_1=\delta_2=1$, 
\begin{equation*}
f(\xi)= \frac 1{2\sqrt{\pi}}\int_{-\infty}^\xi e^{-s^2/4}ds.
\end{equation*}
Hence $f(0)=\frac 12$ and $\gamma(\phi)=f^\prime(0)=\frac 1{2\sqrt{\pi}}$ in the formula \eqref{asymptotic formula in time} by Theorem \ref{th: asymptotic formula in time}. This also corresponds to the formula:
\begin{equation*}
\lim_{t \to 0^+}t^{-\frac 12}\left(u(x,t)-\frac {\sqrt{\sigma_+}}{\sqrt{\sigma_+}+\sqrt{\sigma_-}}\right) =\frac {\sqrt{\sigma_+}\sqrt{\sigma_-}}{\sqrt{\pi}(\sqrt{\sigma_+}+\sqrt{\sigma_-})} (N-1)H(x)
\end{equation*}
with $\sigma_+=\sigma_-=1$, which is given by \cite[Theorem 1.1]{Sa2025} for the linear diffusion equation $u_t=\mbox{div}(\sigma \nabla u)$, where $\sigma = \sigma(x)\ \left(x \in \mathbb R^N\right)$ is a conductivity distribution of the whole medium satisfying
\begin{equation*}
0 < \mu \le \sigma(x) \le M\ \mbox{ for every } x \in \mathbb R^N, \mbox{ and }\sigma(x) = \begin{cases} \sigma_+ \!\! &\mbox{ if }\ x \in B_\rho(p) \cap \Omega,\\  \sigma_-  \!\! &\mbox{ if }\  x \in B_\rho(p) \setminus \Omega,
\end{cases}
\end{equation*}
for four positive constants $\mu, M, \sigma_+ $ and $\sigma_-$. 

The second purpose of the present paper  is to give some applications  of Theorem \ref{th: asymptotic formula in time} to some overdetermined problems. 
Under the assumption on $\phi$ given by Theorem \ref{th: asymptotic formula in time},
the following corollary follows immediately from Theorem \ref{th: asymptotic formula in time}.


\begin{corollary}
\label{constant mean curvature}  Assume that $\phi \in C^3(\mathbb R)$ and $\max\limits_{-1\le s\le2}\{|\phi^{\prime\prime\prime}(s)|, |\phi^{\prime\prime}(s)|\} \le c$ for the constant $c$  given by Theorem \ref{th: asymptotic formula in time}. Let $u$ be the solution of \eqref{Cauchy problem}. Suppose that $\Gamma$ is a connected component of $\partial\Omega, \ \Gamma$ is of class $C^2$ and there exists a function $a : (0, +\infty) \to (0, +\infty)$ satisfying
\begin{equation}
\label{stationary level}
u(x, t) = a(t)\ \mbox{ for every }\ (x, t) \in \Gamma \times (0, +\infty).
\end{equation}
Then the mean curvature of $\Gamma$ must be constant.
\end{corollary}
In Corollary \ref{constant mean curvature}, the overdetermined condition \eqref{stationary level} means that $\Gamma$ is a stationary level surface.
The following two theorems are examples of the application of Corollary \ref{constant mean curvature}  and Theorem \ref{th: asymptotic formula in time} to the overdetermined problems for nonlinear diffusion equations where the surface with initial discontinuities includes a  stationary level surface of class $C^2$.


\begin{theorem}
\label{th:hyperplane}
 Assume that $\phi \in C^3(\mathbb R)$ and $\max\limits_{-1\le s\le2}\{|\phi^{\prime\prime\prime}(s)|, |\phi^{\prime\prime}(s)|\} \le c$ for the constant $c$  given by Theorem \ref{th: asymptotic formula in time}.
Let $\Omega \subset \mathbb R^N$ be a domain given by
$$
\Omega = \{ x \in \mathbb R^N :\, x_N > \varphi(x_1,\dots,x_{N-1}) \},
$$
where $\varphi$ is a $C^2$ function on $\mathbb R^{N-1}$. Let $u$ be the solution of \eqref{Cauchy problem}. Suppose that there exists a function $a : (0, +\infty) \to (0, +\infty)$ satisfying
\begin{equation}
\label{stationary level on entire graph}
u(x, t) = a(t)\ \mbox{ for every }\ (x, t) \in \partial\Omega \times (0, +\infty).
\end{equation}
Then, if $N=2$, $\partial\Omega$ must be a straight line, if $3 \le N \le 8$, $\partial\Omega$ must be a hyperplane, and if $N\ge 9$ and $\nabla \varphi$ is bounded, $\partial\Omega$ must be a hyperplane.
\end{theorem}


\begin{theorem}
\label{th:sphere theorem}
 Assume that $\phi \in C^3(\mathbb R)$ and $\max\limits_{-1\le s\le2}\{|\phi^{\prime\prime\prime}(s)|, |\phi^{\prime\prime}(s)|\} \le c$ for the constant $c$  given by Theorem \ref{th: asymptotic formula in time}.
Let $\Omega \subset \mathbb R^N$ be an open set, whose boundary $\partial\Omega$ is of class $C^0$ and has a bounded connected component $\Gamma$ of class $C^2$. Suppose that
either the inside of $\Gamma$ or the outside of $\Gamma$ is included in  one of the two sets, $\Omega$ and $\mathbb R^N\setminus\overline{\Omega}$.
Let  $u$ be the solution of \eqref{Cauchy problem}. Suppose that there exists a function $a : (0, +\infty) \to (0, +\infty)$ satisfying
\begin{equation}
\label{stationary level surface with one C2 bounded part}
u(x, t) = a(t)\ \mbox{ for every }\ (x, t) \in \partial\Omega \times (0, +\infty).
\end{equation}
Then $\partial\Omega$ must be a sphere.
\end{theorem}

The rest of the paper is organized as follows. 
Section \ref{section_nonlinear diffusion equations} is devoted to introducing a class of nonlinear diffusion equations and self-similar solutions which  play a key role later. In section \ref{section_sub_super solutions for bounded domains}, we construct sub- and supersolutions for domains with bounded $C^2$ boundaries. The Gaussian bounds due to Aronson for the fundamental solution of $ w_t=\mbox{div}(\phi^\prime(v)\nabla w)$ for each fixed function $v=v(x,t)$ and the self-similar solutions given in the previous section are useful. 

Section \ref{section_Proof_of_Theorem 1.1} is devoted to the proof of Theorem \ref{th: asymptotic formula in time} and it consists of seven subsections. Let us mention a significant difference between the present paper and the previous paper \cite{Sa2025}. In \cite{Sa2025}, by applying the Gaussian bounds to the difference of two solutions because of linearity, we directly reduce the problem to the case where $\Omega$ is a bounded $C^2$ domain, but in the present paper, we cannot do the same because of nonlinearity. Nevertheless, in subsection \ref{subsection A local key estimate}, we can estimate the solution locally from above and below by modified sub- and supersolutions in Lemma \ref{A local key estimate from above and below}, which is different from \cite[Lemma 3.1]{Sa2025}.
Lemma \ref{A local key estimate from above and below} in the present paper is local, but \cite[Lemma 3.1]{Sa2025} is global. 

Once this key lemma, Lemma \ref{A local key estimate from above and below}, is established, we can follow the steps in \cite[section 4]{Sa2025}. We employ in principle the blow-up arguments due to Ni-Takagi \cite{NT1993} which succeeded to extract the mean curvature of the boundary from the asymptotic behavior of the least-energy solutions of a singularly perturbed semilinear elliptic Neumann problem.
 We first introduce a principal coordinate system at each point on the boundary $\partial\Omega\cap B_\rho(p)$ together with the mean curvature of $\partial\Omega$.  Next we straighten $\partial\Omega$ locally and then introduce the standard parabolic scaling with a small positive parameter $\varepsilon$ for our blow-up arguments. To perform further asymptotics for the scaled solution $v^\varepsilon$ as $\varepsilon \to 0^+$, we utilize the standard interior H\"older estimates \cite{Li1996} for the second order parabolic equations of divergence form to guarantee  compactness to employ the blow-up arguments. 
 (In \cite{Sa2025}, because of the discontinuity of the diffusion coefficient on the interface, Dong's interior estimates \cite{Dong2012} are utilized.) The most important is the analysis of the second term $S^\varepsilon$ given by \eqref{function for the second mean curvature term} in the asymptotic behavior of $v^\varepsilon$ as $\varepsilon \to 0^+$. Both Lemma \ref{the zeroth-order approximation} coming from Lemma \ref{A local key estimate from above and below} and the preliminary estimate \eqref{Holder norm of coefficients for the equation} play a key role in utilizing the interior estimates. Once interior estimates are obtained, the uniqueness of the limit functions guaranteed by the comparison principle for the Cauchy problems plays a key role. Here,  the limit function $v^*$ is explicit (see (ii) in Lemma \ref{the zeroth-order approximation} with \eqref{one-dimensional self-similar solution}) and the other limit function $S^*$ is uniquely determined with some estimates \eqref{estimate of function G} from above and below. 
 Finally, the uniform convergence on the boundary  in Theorem \ref{th: asymptotic formula in time} also follows from the blow-up arguments.
In section \ref{section_Applications} we prove Theorems  \ref{th:hyperplane} and \ref{th:sphere theorem}. The proof of Theorem \ref{th:hyperplane}
is the same as in \cite{Sa2025}, but the proof of Theorem \ref{th:sphere theorem} becomes simpler because of the unique continuation theorems
for parabolic operators ( \cite[Theorem 1, p. 37]{EF2003} or its stronger versions \cite[Theorem 3, p. 1599]{F2003} and  \cite[Theorem 1, p. 500]{AV2003}).

\setcounter{equation}{0}
\setcounter{theorem}{0}

\section{A class of nonlinear diffusion equations and one-dimensional self-similar solutions}
\label{section_nonlinear diffusion equations}

Let $\phi : \mathbb R \to \mathbb R$ satisify
\begin{equation}
\label{nonlinear diffusion C3}
\phi \in C^3(\mathbb R),\ \phi(0)=0,\ \mbox{ and }\ 0 < \delta_1 \le \phi^\prime(s) \le \delta_2 \ \mbox{ for } s \in \mathbb R,
\end{equation}
where $\delta_1, \delta_2$ are positive constants. The class of nonlinear diffusion equations we deal with in the present paper is obtained by another additional restriction
\begin{equation}\label{another restriction}
\max_{-1\le s\le2}\{|\phi^{\prime\prime\prime}(s)|, |\phi^{\prime\prime}(s)|\} \le c,
\end{equation}
where $c$ is a positive constant depending only on $\delta_1$ and $\delta_2$ to be given in Lemma \ref{upper and lower barriers} in section \ref{section_sub_super solutions for bounded domains}. 
Let $f=f(\xi)$ for $\xi \in \mathbb R$ be the unique solution of the problem stated in Theorem \ref{th: asymptotic formula in time}:
\begin{equation}\label{for selfsimilar solutions}
\begin{cases}
&\left(\phi^\prime(f)f^\prime\right)^\prime+\frac12\xi f^\prime=0 \mbox{ and } f^\prime>0\ \mbox{ in } \mathbb R,\\
 &f(-\infty)=0\ \mbox{ and }\  f(+\infty)=1.
 \end{cases}
 \end{equation}
 The existence and uniqueness of $f$ follows along the same lines as in \cite[Appendix A, pp. 249--251]{MS2012} with the aid of a result in \cite{AP1974}.
 Let us define the function $\Psi =\Psi(\ell,t)$ by
 \begin{equation}\label{one-dimensional self-similar solution}
 \Psi(\ell,t)=f(t^{-\frac12}\ell)\ \mbox{ for } \ell \in \mathbb R\mbox{ and } t > 0.
 \end{equation}
Then $\Psi$ is the unique bounded solution of  the one-dimensional Cauchy problem:
\begin{equation}\label{one-dimensional Cauchy}
\Psi_t=\partial_\ell^2\phi(\Psi)\ \mbox{ in } \mathbb R\times (0, +\infty)\ \mbox{ and }\ \Psi =\mathcal X_{(0, +\infty)}\ \mbox{ on } \mathbb R \times \{0\}.
\end{equation}
$\Psi$ is called a self-similar solution.
When $\phi(s)\equiv s$ in particular, $u$ satisfies the heat equation $u_t=\partial_\ell^2u$ and $f=f(\xi)$ is explicitly solved. Thus we set, for $\ell \in \mathbb R, t > 0$ and $\xi \in \mathbb R$, 
\begin{equation}
\label{the heat equation}
u^*(\ell,t) = f_*(t^{-\frac12}\ell)\ \mbox{ and } f_*(\xi) =  \frac 1{2\sqrt{\pi}}\int_{-\infty}^\xi e^{-\eta^2/4}d\eta.
\end{equation}
Then we notice that $u^*$ and $f_*$ satisfy
\begin{eqnarray}
&&u^*_t=\partial_\ell^2u^*\ \mbox{ in } \mathbb R\times (0, +\infty)\ \mbox{ and }\ u^* =\mathcal X_{(0, +\infty)}\ \mbox{ on } \mathbb R \times \{0\},\label{heat equation}\\
&&f_*^{\prime\prime}+\frac 12\xi f_*^\prime=0,\ 0 < f_*^\prime \le \frac 1{2\sqrt{\pi}} \mbox{ and }\ 0 < f_* < 1\ \mbox{ in } \mathbb R, \label{equation of selfsimilar solution and monotomicity}\\
&&f_*(-\infty) =0,\ f_*(0) = \frac 12,\ f_*(+\infty) =1,\ f_*^\prime(0)=\frac 1{2\sqrt{\pi}} \mbox{ and }  f_*^\prime(\pm\infty)=0.\label{properties of selfsimilar}
\end{eqnarray}
Therefore, setting $W(\ell,t)=u^*(\ell,\delta_2t)$ for  $\ell \in \mathbb R$ and $ t > 0$ yields that the function $W=W(\ell,t)$ satisfies 
\begin{align}
W_t&=\delta_2\partial_\ell^2W &\mbox{ in } &\mathbb R\times (0, +\infty),\label{Eq of W}\\
W &=\mathcal X_{(0, +\infty)} &\mbox{ on } &\mathbb R \times \{0\}, \label{initial condition}\\
W&=\frac12 &\mbox{ on } &\{0\}\times  (0, +\infty).\label{boundary condition}
\end{align}
Set $U=\phi(\Psi)$ and $V=\phi(1)-\phi(\Psi)$. Then we observe that
\begin{align}
&U_t=\phi^\prime(\Psi)\partial_\ell^2 U\ \mbox{ and } V_t=\phi^\prime(\Psi)\partial_\ell^2 V\ &\mbox{ in } \mathbb R\times (0, +\infty),\\
&U_t > 0\mbox{ and }\partial_\ell^2U>0 \ &\mbox{ in } (-\infty,0)\times  (0, +\infty),\\
&V_t> 0\mbox{ and }\partial_\ell^2V>0 \ &\mbox{ in } (0, +\infty)\times  (0, +\infty),\\
&U=\phi(1)\mathcal X_{(0,+\infty)}\mbox{ and }V=\phi(1)\mathcal X_{(-\infty,0)} &\mbox{ on }\mathbb R \times \{0\}, \\
&U=\phi(f(0)) \mbox{ and } V=\phi(1)-\phi(f(0)) \ &\mbox{ on } \{0\}\times  (0, +\infty).
\end{align}
Hence it follows from the comparison principle that
\begin{align}
&0 < U \le 2\phi(1) W           \!\!\!          &\mbox{ in } &(-\infty,0)\times(0,+\infty),\label{bound for U}\\ 
 &0 < V \le 2\phi(1)(1-W)\!\!\! &\mbox{ in } &(0,\infty)\times(0,+\infty).\label{bound for V}
\end{align}
Namely,
\begin{align}
&0 < \phi(f(\xi)) \le \pi^{-\frac12}\phi(1)\int_{-\infty}^{\delta_2^{-\frac12}\xi} e^{-\eta^2/4}d\eta               &\mbox{ for } & \xi < 0,\label{bound for U in xi}\\ 
 &0 < \phi(1)-\phi(f(\xi)) \le  \pi^{-\frac12}\phi(1)\int^{\infty}_{\delta_2^{-\frac12}\xi} e^{-\eta^2/4}d\eta  &\mbox{ for  } & \xi > 0.\label{bound for V in xi}
\end{align}
Therefore, since $\phi(0)=0$ and $(\phi^{-1})^\prime\le\delta_1^{-1}$, we observe that
\begin{align}
&0 < f(\xi) \le \pi^{-\frac12}\delta_1^{-1}\phi(1)\int_{-\infty}^{\delta_2^{-\frac12}\xi} e^{-\eta^2/4}d\eta               &\mbox{ for } & \xi < 0,\label{bound for f in xi}\\ 
 &0 < 1-f(\xi) \le  \pi^{-\frac12}\delta_1^{-1}\phi(1)\int^{\infty}_{\delta_2^{-\frac12}\xi} e^{-\eta^2/4}d\eta  &\mbox{ for  } & \xi > 0.\label{bound for 1-f in xi}
\end{align}

\begin{proposition}
\label{bounds on f} 
There exists a constant $k=k(\delta_1,\delta_2) \ge 1$ depending only on $\delta_1$ and $\delta_2$ such that if $\max\limits_{0\le s \le 1}|\phi^{\prime\prime}(s)| \le 1$, then $\sup\limits_{\xi \in \mathbb R}\{ f(\xi), f^\prime(\xi), |f^{\prime\prime}(\xi)|\} \le k$.
\end{proposition}

\noindent
{\it Proof.} Since $0 < f < 1$ from \eqref{for selfsimilar solutions}, we may choose $k \ge 1$ with respect to $f$. Let us estimate $f^\prime$. It follows from \eqref{for selfsimilar solutions} that
\begin{equation}\label{maximum of phi prime f prime}
0 < \phi^\prime(f)f^\prime \le \phi^\prime(f(0))f^\prime(0)\ \mbox{ in } \mathbb R.
\end{equation}
Let us introduce a positive function $h=h(\xi)$ for $\xi \in \mathbb R$ by $h(\xi) = \phi(f(\xi))$. Then $h^\prime=\phi^\prime(f)f^\prime > 0, h^\prime(0)=\phi^\prime(f(0))f^\prime(0)$ and
\begin{equation}\label{quotient function}
(\log h^\prime)^\prime = -\frac 1{2\phi^\prime(f)}\xi \ \mbox{ for } \xi \in \mathbb R.
\end{equation}
Integrating \eqref{quotient function} yields that 
\begin{equation}\label{estimate of h prime}
0 < h^\prime(\xi) \le h^\prime(0)\exp\left\{-\frac {\xi^2}{4\delta_2}\right\}\ \mbox{ for  every } \xi \in \mathbb R,
\end{equation}
and hence, $\lim\limits_{\xi \to \pm\infty}h^\prime(\xi) = 0$. On the other hand, by integrating the equation in \eqref{for selfsimilar solutions} we obtain 
\begin{equation}\label{direct integration of original ode}
h^\prime(0) - h^\prime(\xi) = -\frac 12\int_\xi^0\eta f^\prime(\eta)d\eta=\frac12\left(\xi f(\xi)+\int_\xi^0f(\eta)d\eta\right) \mbox{ for every } \xi < 0.
\end{equation}
Thus, with the aid of \eqref{bound for f in xi}, letting $\xi \to -\infty$ gives
\begin{equation}\label{identity for h}
h^\prime(0) = \frac 12\int_{-\infty}^0f(\eta) d\eta.
\end{equation}
On the other hand, since $\delta_1\le \phi^\prime\le \delta_2$ and $\phi(0)=0$, we notice that
\begin{equation}\label{bound of phi f}
0 < h=\phi(f) <\phi(1) \le \delta_2\ \mbox{ in }\mathbb R.
\end{equation}
Therefore, combining \eqref{maximum of phi prime f prime}, \eqref{identity for h}, \eqref{bound for f in xi} and \eqref{bound of phi f} with \eqref{estimate of h prime} yields that there exists a positive constant $k_1$ depending only on $\delta_1$ and $\delta_2$ satisfying
\begin{equation}\label{estimates for phi prime f prime and f prime}
0 < \phi^\prime(f)f^\prime \le k_1\exp\left\{-\frac {\xi^2}{4\delta_2}\right\} \mbox{ and } 0< f^\prime \le  {k_1}\exp\left\{-\frac {\xi^2}{4\delta_2}\right\}\ \mbox{ for every }\xi \in \mathbb R.
\end{equation}
Moreover, by the equation  in \eqref{for selfsimilar solutions}, we have
\begin{equation}\label{estimate second derivative of f}
f^{\prime\prime}(\xi)= -\frac 1{\phi^\prime(f(\xi))}\left\{\phi^{\prime\prime}(f(\xi)) (f^\prime(\xi))^2+\frac 12\xi f^\prime(\xi)\right\}\ \mbox{ for every }\xi \in \mathbb R.
\end{equation}
Finally, combining \eqref{estimates for phi prime f prime and f prime} and  \eqref{estimate second derivative of f} with the assumption that $\max\limits_{0\le s \le 1}|\phi^{\prime\prime}(s)| \le 1$ yields the conclusion. \qed

\setcounter{equation}{0}
\setcounter{theorem}{0}

\section{Sub- and supersolutions for domains with bounded $C^2$ boundaries}
\label{section_sub_super solutions for bounded domains}

Let $D$ be a domain with bounded $C^2$ boundary $\partial D$ in $\mathbb R^N$ where $N\ge 2$ and let $v=v(x,t)$ be the unique bounded solution  of \eqref{Cauchy problem} where $\Omega$ is replaced with $D$.  Namely, we have
\begin{equation}
\label{Cauchy problem for bounded D}
v_t=\Delta \phi(v)\ \mbox{ in }\ \mathbb R^N\times(0,+\infty)\ \mbox{ and }\ v =\mathcal{X}_{D}\ \mbox{ on }\ \mathbb R^N\times\{0\},
\end{equation}
where $\mathcal X_D$ denotes the characteristic function of the set $D$.
The purpose of this section is to construct sub and super solutions for  this problem in a neighborhood of $\partial D$.  This is a key throughout this paper.
We introduce the signed distance function $d^*=d^*(x)$ of $x \in \mathbb R^N$ to the boundary $\partial D$ by
\begin{equation}
\label{signed distance}
d^*(x) = \begin{cases} \mbox{ dist}(x, \partial D)  \!\!&\mbox{ if }\  x \in  D,\\  - \mbox{dist}(x, \partial D)  \!\! &\mbox{ if }\  x \in \mathbb R^N \setminus D.
\end{cases}
\end{equation}
Since $\partial D$ is compact and of class $C^2$, there exists a number $\delta_0 > 0$ such that $d^*(x)$ is of class $C^2$ on the closure of a bounded neighborhood $\mathcal N$ of $\partial D$ given by
\begin{equation}
\label{Neighborhood of the boundary} 
\mathcal N = \{ x \in \mathbb R^N\, :\, -\delta_0 < d^*(x) < \delta_0 \}.
\end{equation}

Define the function $\psi=\psi(x, t)$ by 
\begin{equation}\label{approximate solution near the boundary}
\psi(x,t) = f\left(t^{-\frac12}d^*(x)\right) \mbox{ for } (x,t) \in \mathbb R^N\times(0,+\infty),
\end{equation}
where the function $f$ is given in \eqref{for selfsimilar solutions}. Since $|\nabla d^*| = 1$ on $\overline{\mathcal N}$, a straightforward computation yields that
\begin{equation}\label{diffusion equation for approximate solution}
\psi_t- \Delta \phi(\psi) =- t^{-\frac 12}\phi^\prime(f)f^\prime\Delta d^*\  \mbox{ in }\  \mathcal N\times (0, + \infty).
\end{equation}
For a positive constant $\Lambda$ to be chosen later, we define the two functions $w^\pm=w^\pm(x,t)$ by
\begin{equation}\label{definition of sub- and supersolutions}
w^\pm(x,t) =\psi(x,t)\pm \Lambda\sqrt{t}\ \mbox{ for }(x,t) \in \overline{\mathcal N}\times (0,+\infty).
\end{equation}
Let us calculate $w^\pm_t-\Delta\phi(w^\pm)$. Since
$$
\phi^\prime(w^\pm)=\phi^\prime(\psi) \pm\Lambda\sqrt{t} \int_0^1\phi^{\prime\prime}(\psi\pm\theta\Lambda\sqrt{t}) d\theta,
$$
we have from \eqref{diffusion equation for approximate solution} that for $(x,t) \in \overline{\mathcal N}\times (0,+\infty)$
\begin{align*}
w^\pm_t-\Delta \phi(w^\pm)&=\psi_t\pm\frac12 \Lambda t^{-\frac12}- \mbox{ div}(\phi^\prime(w^\pm)\nabla\psi)\\
&=- t^{-\frac 12}\phi^\prime(f)f^\prime\Delta d^*\pm\frac12 \Lambda t^{-\frac12}  \mp\Lambda\sqrt{t} \mbox{ div}\left(\int_0^1\phi^{\prime\prime}(\psi\pm\theta\Lambda\sqrt{t}) d\theta\nabla\psi\right).
\end{align*}
Hence, using \eqref{approximate solution near the boundary} with $|\nabla d^*| = 1$ yields that
\begin{align}
&\sqrt{t}\left\{w^\pm_t-\Delta \phi(w^\pm)\right\}\nonumber\\
&\ =-\phi^\prime(f)f^\prime\Delta d^*\label{useful identity for sub- and supersolutions}\\
&\quad\ \pm\Lambda\left\{\frac12-(f^\prime)^2\int_0^1\phi^{\prime\prime\prime}(\psi\pm\theta\Lambda\sqrt{t}) d\theta-(f^{\prime\prime}+\sqrt{t}f^\prime\Delta d^*)\int_0^1\phi^{\prime\prime}(\psi\pm\theta\Lambda\sqrt{t}) d\theta\right\}\nonumber.
\end{align}
Next, we use the Gaussian bounds for the fundamental solutions of diffusion equations due to
Aronson \cite[Theorem 1, p.891]{Ar1967}(see also \cite[p.328]{FS1986}). Let $g = g(x,t;\xi,\tau)$ be the fundamental solution of $w_t=\mbox{div}(\phi^\prime(v)\nabla w)$. Then there exist two positive constants $\kappa< \mathcal K$ depending only on $\delta_1, \delta_2$ and $N$ such that
\begin{equation}
\label{Gaussian bounds}
\kappa\, (t-\tau)^{-\frac N2}e^{-\frac{|x-\xi|^2}{\kappa (t-\tau)}}\le g(x,t; \xi,\tau) \le \mathcal K\, (t-\tau)^{-\frac N2}e^{-\frac{|x-\xi|^2}{\mathcal K (t-\tau)}}
\end{equation}
 for all $(x,t), (\xi,\tau) \in \mathbb R^N\times[0,+\infty)$ with $t > \tau$. Notice that $v$ satisfies
 \begin{equation}
 \label{expression of the solution}
 v(x,t)=\int\limits_D g(x,t; \xi,0)d\xi\ \mbox{ for } (x,t) \in \mathbb R^N\times(0, +\infty).
 \end{equation}
We observe from \eqref{Gaussian bounds}, \eqref{expression of the solution}, the definition of $\psi$, \eqref{bound for f in xi} and \eqref{bound for 1-f in xi} with \eqref{bound of phi f} that there exists two positive constants $A$ and $a$ depending only on $\delta_0, \delta_1, \delta_2$ and $N$ such that
\begin{eqnarray}
&0 < v(x,t) \le Ae^{-\frac at}\ &\mbox{ for every } (x,t) \in \left(\partial \mathcal N\setminus D\right)\times (0,+\infty),\label{exponential decay outside D for u}\\
&0 < \psi(x,t) \le Ae^{-\frac at}\ &\mbox{ for every } (x,t) \in \left(\partial \mathcal N\setminus D\right)\times (0,+\infty),\label{exponential decay outside D for psi}\\
&0 < 1-v(x,t) \le Ae^{-\frac at}\ &\mbox{ for every } (x,t) \in \left(\partial \mathcal N\cap D\right)\times (0,+\infty),\label{exponential decay inside D for u}\\
&0 < 1-\psi(x,t) \le Ae^{-\frac at}\ &\mbox{ for every } (x,t) \in \left(\partial \mathcal N\cap D\right)\times (0,+\infty).\label{exponential decay inside D for psi}
\end{eqnarray}


\begin{lemma}
\label{upper and lower barriers} There exists a positive constant $c=c(\delta_1, \delta_2) \le 1$ depending only on $\delta_1, \delta_2$ such that if $\phi \in C^3(\mathbb R)$ satisfies $\max\limits_{-1\le s\le2}\{|\phi^{\prime\prime\prime}(s)|, |\phi^{\prime\prime}(s)|\} \le c$, then the solution $v$ of problem \eqref{Cauchy problem for bounded D} satisfies
\begin{equation}
\label{new pointwise estimates}
\psi(x,t)-\Lambda\sqrt{t} \le v(x,t) \le \psi(x,t)+\Lambda\sqrt{t}\ \mbox{ for every } (x,t) \in \mathcal N \times (0,+\infty),
\end{equation}
for some positive constant $\Lambda$ depending only on $\delta_0, \delta_1, \delta_2, N$ and $\max\limits_{\overline{\mathcal N}}|\Delta d^*|$.
\end{lemma}

\noindent
{\it Proof.\ } Since $0 < v <1$ and $0 < \psi < 1$, we have
\begin{equation}
\label{new pointwise estimates for large time}
w^{-}(x,t)<0< v(x,t) < 1 < w^+(x,t)\ \mbox{ for every } (x,t) \in \mathcal N \times (\Lambda^{-2},+\infty).
\end{equation}
Thus, it suffices to consider the case where $0 < t \le \Lambda^{-2}$. Then the identity \eqref{useful identity for sub- and supersolutions} is rewritten as
\begin{align}
&\sqrt{t}\left\{w^\pm_t-\Delta \phi(w^\pm)\right\}\nonumber\\
&\ =\left\{-\phi^\prime(f) \mp\Lambda\sqrt{t}\int_0^1\phi^{\prime\prime}(\psi\pm\theta\Lambda\sqrt{t}) d\theta\right\} f^\prime\Delta d^*\label{useful identity for sub- and supersolutions 2nd}\\
&\quad\ \pm\Lambda\left\{\frac12-(f^\prime)^2\int_0^1\phi^{\prime\prime\prime}(\psi\pm\theta\Lambda\sqrt{t}) d\theta-f^{\prime\prime}\int_0^1\phi^{\prime\prime}(\psi\pm\theta\Lambda\sqrt{t}) d\theta\right\}\nonumber.
\end{align}
Proposition \ref{bounds on f} gives a positive constant $c=c(\delta_1, \delta_2) \le 1$ depending only on $\delta_1, \delta_2$ such that if $\phi \in C^3(\mathbb R)$ satisfies $\max\limits_{-1\le s\le2}\{|\phi^{\prime\prime\prime}(s)|, |\phi^{\prime\prime}(s)|\} \le c$, then for every $0 < t \le \Lambda^{-2}$ 
\begin{equation}\label{finding constant c}
\frac12-(f^\prime)^2\int_0^1\phi^{\prime\prime\prime}(\psi\pm\theta\Lambda\sqrt{t}) d\theta-f^{\prime\prime}\int_0^1\phi^{\prime\prime}(\psi\pm\theta\Lambda\sqrt{t}) d\theta \ge \frac 13,
\end{equation}
and moreover, for every $0 < t \le \Lambda^{-2}$, we observe that
\begin{equation}\label{finding constant c 2}
\left|\left\{-\phi^\prime(f) \mp\Lambda\sqrt{t}\int_0^1\phi^{\prime\prime}(\psi\pm\theta\Lambda\sqrt{t}) d\theta\right\} f^\prime\Delta d^*\right| \le (\delta_2+c)f^\prime\max\limits_{\overline{\mathcal N}}|\Delta d^*|.
\end{equation}
Let $\max\limits_{-1\le s\le2}\{|\phi^{\prime\prime\prime}(s)|, |\phi^{\prime\prime}(s)|\} \le c$ and $0 < t \le \Lambda^{-2}$. Then, combining \eqref{useful identity for sub- and supersolutions 2nd}, \eqref{finding constant c}, \eqref{finding constant c 2} with $f^\prime \le k$ from Proposition  \ref{bounds on f} yields that if $\Lambda \ge 3  (\delta_2+c)k\max\limits_{\overline{\mathcal N}}|\Delta d^*|$, 
\begin{equation}\label{differential inequalities}
\pm\left\{w^\pm_t-\Delta \phi(w^\pm) \right\} \ge 0\ \mbox{ in } \mathcal N \times (0, \Lambda^{-2}].
\end{equation}
On the other hand, with the aid of inequality $se^{-as^2}\le (2ea)^{-\frac12}$ for $s >0$ and $a > 0$,  by combining \eqref{exponential decay outside D for u}, \eqref{exponential decay outside D for psi}, \eqref{exponential decay inside D for u} and \eqref{exponential decay inside D for psi}, we observe that if $\Lambda \ge A(2ea)^{-\frac12}$
\begin{equation}\label{on the boundary}
w^-\le v\le w^+\ \mbox{ on }\partial \mathcal N\times(0, \Lambda^{-2}].
\end{equation}
Hence we set $\Lambda=\max\{3  (\delta_2+c)k\max\limits_{\overline{\mathcal N}}|\Delta d^*|, A(2ea)^{-\frac12}\}$. Notice that $\Lambda$ depends only on $\delta_0, \delta_1, \delta_2, N$ and $\max\limits_{\overline{\mathcal N}}|\Delta d^*|$.
Therefore, since $v=w^+=w^-=\mathcal X_D$ on $\mathcal N\times\{0\}$, it follows from \eqref{differential inequalities}, \eqref{on the boundary} and the comparison principle that
$$
w^-\le v\le w^+\ \mbox{ in }\mathcal N\times(0, \Lambda^{-2}].
$$
Consequently, we get the conclusion by combining this with \eqref{new pointwise estimates for large time}. \qed

\setcounter{equation}{0}
\setcounter{theorem}{0}

\section{Proof of Theorem \ref{th: asymptotic formula in time}}
\label{section_Proof_of_Theorem 1.1}

Since the equation we consider is nonlinear, we cannot directly reduce our problem to the case where $\Omega$ is a bounded $C^2$ domain as in \cite[Section 2, pp. 224--225]{Sa2025}.
However, we will be able to circumvent this inconvenience by using Lemma \ref{upper and lower barriers} for two different $C^2$ domains with bounded boundaries, where one is a bounded domain and the other is an exterior domain. See subsection \ref{subsection A local key estimate} below for details.

As in \cite{Sa2025}, we employ in principle  the blow-up arguments due to Ni-Takagi \cite{NT1993} which succeeded to extract the mean curvature of the boundary from the asymptotic behavior of the least-energy solutions of a singularly perturbed semilinear  elliptic Neumann problem.
By straightening the boundary and introducing the scaling related to the small parameter $\varepsilon$ in their singularly perturbed problem, they find the mean curvature of the boundary in the elliptic equation satisfied by the first-order approximation in $\varepsilon$ as $\varepsilon \to 0^+$ (see \cite[(2.4), p.251 and (2.8), p.252]{NT1993}). 

Here we straighten $\partial\Omega$ locally and introduce the standard parabolic scaling with a small positive parameter $\varepsilon$. Then we can find the mean curvature of $\partial\Omega$ in the inhomogeneous diffusion equation \eqref{limit equation 2nd} satisfied by the first-order approximation $S^*$ in $\varepsilon$  as $\varepsilon \to 0^+$ of the scaled solution $v^\varepsilon$. Our proof consists of seven steps.

\subsection{A local key estimate of the solution near the boundary $\partial\Omega$}
\label{subsection A local key estimate}
Choose  $0 < r < \rho$ arbitrarily and set $r_1=\frac {r+\rho}2$. Since $\partial\Omega\cap B_{\rho}(p)$ is of class $C^2$, we may find two bounded domains $\Omega_+, \Omega_-$ of class $C^2$ satisfying
\begin{equation}\label{two bounded domains}
\Omega_+\subset \Omega\cap B_\rho(p),\ \Omega_-\subset (\mathbb R^N\setminus\overline{\Omega})\cap B_\rho(p)\ \mbox{ and }\partial\Omega\cap \overline{B_{r_1}(p)} \subset \partial \Omega_+\cap\partial \Omega_-.
\end{equation}
Setting $D_1=\Omega_+$ and $D_2=\mathbb R^N\setminus \overline{\Omega_-}$ yields that
\begin{equation}\label{monotonicity of three domains}
D_1 \subset \Omega \subset D_2\ \mbox{ and }\partial D_1 \cap \partial D_2 \cap \partial\Omega\supset\partial\Omega\cap \overline{B_{r_1}(p)}.
\end{equation}
Let $v_j=v_j(x,t)\ (j=1,2)$ be the unique bounded solutions of \eqref{Cauchy problem for bounded D} where $D=D_j$. Then, since $\mathcal X_{D_1}\le\mathcal X_{\Omega}\le\mathcal X_{D_2}$, 
it follows from the comparison principle that the solution $u$ of \eqref{Cauchy problem} satisfies
\begin{equation}\label{estimate from above and below}
v_1 \le u \le v_2\ \mbox{ in }\ \mathbb R^N\times (0, +\infty).
\end{equation}
Since $\partial D_j\ (j=1,2)$ are bounded and of class $C^2$, we can apply Lemma \ref{upper and lower barriers} to both domains.
Set  $r_2=\frac {3r+\rho}4$. By choosing the number $\delta_0$ for $D_1$ and for $D_2$ sufficiently small (see \eqref{Neighborhood of the boundary} for the number $\delta_0$), we may allow $\partial D_1$ and $ \partial D_2$ to have the common number $\delta_0$ and the common signed distance function $d^*$ in the intersection of the tubular neighborhoods of $\partial D_1$ and $\partial D_2$ in $\overline{B_{r_2}(p)}$. Thus, by virtue of \eqref{estimate from above and below} and Lemma \ref{upper and lower barriers}, we have
\begin{lemma}\label{A local key estimate from above and below}
There exist $\delta_0 >0$ and $\Lambda > 0$ such that
\begin{equation}\label{A local key estimate}
\psi(x,t) -\Lambda \sqrt{t} \le u(x,t) \le \psi(x,t) + \Lambda\sqrt{t} \mbox{ for every } (x,t) \in (\mathcal N \cap B_{r_2}(p)) \times (0,+\infty),
\end{equation}
where $d^*(x)$ is the signed distance function to the boundary $\partial \Omega$, $\psi$ is given by \eqref{approximate solution near the boundary} and  $\mathcal N$ is given by \eqref{Neighborhood of the boundary}.
\end{lemma}


\subsection{Introducing a principal coordinate system at each point on $\partial\Omega\cap \overline{B_{r}(p)}$}
\label{subsection principal coordinate system}
 Let $q \in \partial\Omega\cap \overline{B_r(p)} (\subset B_\rho(p))$ arbitrarily. 
  For each point $q$, let us introduce a principal coordinate system with the origin $0$ at $q\in\partial\Omega\cap \overline{B_r(p)}$. Since $\partial\Omega\cap B_{\rho}(p)$ is of class $C^2$, we may choose a positive number $r_0$ with $r_0 < r_2-r (=\frac{\rho-r}4)$, which is independent of $q\in \partial\Omega\cap \overline{B_r(p)}$, such that there exists a function $\varphi \in C^2(\mathbb R^{N-1})$ satisfying
$$
B_{r_0}(q)\cap\Omega=\{  x \in B_{r_0}(0) : x_N < \varphi(\hat{x}) \} \mbox{ and } B_{r_0}(q)\cap\partial\Omega=\{ x \in B_{r_0}(0) : x_N=\varphi(\hat{x})\},
$$
where $\delta_0 $ is the positive constant given in Lemma \ref{A local key estimate from above and below}, $\hat{x}=(x_1,\dots,x_{N-1})$ for $x \in \mathbb R^N$, and the norm $\Vert\varphi\Vert_{C^2(\mathbb R^{N-1})}$  is bounded in $q \in \partial\Omega\cap \overline{B_r(p)}$. Note that $B_{r_0}(q) \subset B_{r_2}(p)$. 
Without loss of generality, we may suppose that 
\begin{equation}
\label{principal curvatures}
q=0,  \varphi(0)=0, \nabla\varphi(0)=0\mbox{ and }\varphi(\hat{x})=\frac 12\sum_{j=1}^{N-1}\kappa_j x_j^2 + o(|\hat{x}|^2)\ \mbox{ as } \hat{x} \to 0,
\end{equation}
where $\kappa_j\ (j=1,\dots,N\!-\!1)$ are the principal curvatures of $\partial\Omega$ at $q \in \partial\Omega\cap \overline{B_r(p)}$ with respect to the outward normal direction to $\partial\Omega$.  Note that the mean curvature $H(q)$ of
 $\partial\Omega$ at $q\in\partial\Omega\cap \overline{B_r(p)}$ is given by
 $$
 H(q)=H(0)=\frac 1{N-1}\sum\limits_{j=1}^{N-1}\kappa_j.
 $$ 
 
 
 \subsection{Straightening $\partial\Omega\cap B_{r_0}(q)$ and scaling for blow-up arguments}
 \label{subsection straightening the interface and scaling}
 As used in the boundary regularity theory for elliptic partial differential equations in \cite{ACM2018, E2010}, let us introduce the coordinate transformation $T : (\hat{x}, x_N) \mapsto (\hat{y}, y_N)=(\hat{x}, \varphi(\hat{x})-x_N)$ which straightens $B_{r_0}(q)\cap\partial\Omega$.  Then,  the Jacobian matrix $\nabla T$ is triangular with $\mbox{det}(\nabla T)=-1$,  its inverse transformation is given by $T^{-1}:  (\hat{y}, y_N) \mapsto (\hat{x}, x_N)=(\hat{y},\varphi(\hat{y})-y_N)$, and $\nabla T^{-1}$ is also triangular with $\mbox{det}(\nabla T^{-1})=-1$. We may put 
 $$
 \mathcal B=T(B_{r_0}(0)),\ T(B_{r_0}(0)\cap\Omega)=\{ y \in \mathcal B :  y_N > 0\} \mbox{ and }T(B_{r_0}(0)\cap\partial\Omega)= \{ y \in \mathcal B : y_N = 0\}.
 $$
 Then,  by setting for $(y,t) \in \mathcal B\times [0, +\infty)$ 
$$
 v=v(y,t) = u(T^{-1}y,t)\ \mbox{ and } \mathcal X_T= \mathcal X_T(y_N) = \mathcal X_\Omega(T^{-1}y),
 $$
we see from \eqref{Cauchy problem} and \eqref{positive values} that
\begin{eqnarray}
 &0 < v < 1\ \mbox{ and }\ v_t=\mbox{div}(\phi^\prime(v) A\nabla v)\ &\mbox{ in } \mathcal B \times (0,+\infty), \label{weak solution after straightening}\\
 &v= \mathcal X_{T} &\mbox{ on } \mathcal B \times \{0\}, \label{initial data after straightening}
 \end{eqnarray}
where  $A=A(\hat{y})$ is the symmetric $N\times N$ matrix given by
\begin{equation}
\label{symmetric uniformly elliptic matrix}
A(\hat{y}) = \left[\begin{array}{c|c} I_{N-1}&\nabla \varphi(\hat{y})\\ \hline
{}^t(\nabla\varphi(\hat{y}))&1+|\nabla\varphi(\hat{y})|^2
\end{array}\right]\ \mbox{ and }\ A(0)=I_N
\end{equation}
 with the identity matrix $I_{k}$ of size $k \in \mathbb N$ and the transpose ${}^t(\nabla\varphi(\hat{y}))$ of  $\nabla \varphi(\hat{y})$. 
 Moreover, by setting 
 \begin{equation}\label{scaled initial data}
 \mathcal X^*=\mathcal X^*(\eta)= \begin{cases} 1  \!\!&\mbox{ if }\  \eta >0,\\  0  \!\!&\mbox{ if }\  \eta \le 0,
\end{cases}
\end{equation}
we observe that 
 $$
 \mathcal X_{T}(y_N) = \mathcal X^*(y_N)\ \mbox{  for }y \in \mathcal B.
 $$
Since the norm $\Vert\varphi\Vert_{C^2(\mathbb R^{N-1})}$  is bounded in $q \in \partial\Omega\cap \overline{B_r(p)}$ and $T0=0$, there exists a number $\rho_0 > 0$ being independent of $q \in \partial\Omega\cap \overline{B_r(p)}$ such that $\overline{B_{\rho_0}(0)} \subset \mathcal B$. Let us  introduce the standard parabolic scaling with a small positive parameter $\varepsilon$ by 
$$
(z,s)=(\varepsilon^{-1}\,y, \varepsilon^{-2}t)\ \mbox{ for }(y, t) \in \overline{B_{\rho_0}(0)} \times [0,+\infty).
$$
Remark that $y \in B_{\rho_0}(0)$ if and only if $z\in B_{\varepsilon^{-1}\rho_0}(0)$. For $(z,s) \in B_{\varepsilon^{-1}\rho_0}(0)\times [0, +\infty)$, we set
\begin{equation}\label{scaled solution and initial data}
v^\varepsilon=v^\varepsilon(z,s) = v\!\left(\varepsilon z, \varepsilon^2 s \right),\  A^*=A^*(\hat{z}) = A\!\left(\varepsilon\hat{z}\right)
 \mbox{ and }\ 
\mathcal X^*=\mathcal X^*(z_N),
\end{equation}
where $\mathcal X^*$ is given by \eqref{scaled initial data} and it is invariant under the scaling.
Then it follows from \eqref{weak solution after straightening} and \eqref{initial data after straightening} that
\begin{eqnarray}
 &0 < v^\varepsilon < 1\ \mbox{ and }\ v^\varepsilon_s=\mbox{div}(\phi^\prime(v^\varepsilon) A^*\nabla v^\varepsilon)\ &\mbox{ in } B_{\varepsilon^{-1}\rho_0}(0)  \times (0,+\infty), \label{weak solution after scaling}\\
 &v^\varepsilon= \mathcal X^*&\mbox{ on }  B_{\varepsilon^{-1}\rho_0}(0)\times \{0\}. \label{initial data after scaling}
 \end{eqnarray}

\subsection{Utilizing the interior H\"older estimates}
\label{subsection interior estimates}
Let us  utilize the regularity theory for the second order parabolic equations of divergence form in \cite{Li1996}.  Note that if $n \in \mathbb N$ and  $\varepsilon \le (n+2)^{-1}$, then $\overline{B_{n\rho_0}(0)} \subset B_{(n+1)\rho_0}(0)\subset\overline{B_{(n+1)\rho_0}(0)}\subset B_{\varepsilon^{-1}\rho_0}(0)$. For each $n \in \mathbb N$, we define a cylinder $Q_n$ in $\mathbb R^{N}\!\times\!(0,+\infty)$ by
\begin{equation}
\label{sequence of cylinders}
Q_n= B_{n\rho_0}(0)\times(n^{-1}\rho_0^2, n^{-1}\rho_0^2+n^2\rho_0^2).
\end{equation}
Hence $\overline{Q_{n}}  \subset  Q_{n+1}\subset\overline{Q_{n+1}}\subset B_{\varepsilon^{-1}\rho_0}(0)\times (0, +\infty)$ and $\bigcup\limits_{n=1}^\infty Q_n = \mathbb R^N\times(0,+\infty)$.
 Let $0 < \alpha < 1$. Then,  with the aid of the interior H\"older estimates \cite[Theorem 6.28, p. 129 and Theorem 4.8, pp. 56--57]{Li1996}, we infer that, for every $n \in \mathbb N$ there exists a positive constant $C_n$, which is independent of $q \in \partial\Omega\cap \overline{B_r(p)}$, such that  for every $0 <\varepsilon \le (n+2)^{-1}$
\begin{equation}
\label{interior estimates}
 \Vert v^\varepsilon\Vert_{C^{\alpha,\alpha/2}(\overline{Q_n})} + \Vert\nabla  v^\varepsilon\Vert_{C^{\alpha,\alpha/2}(\overline{Q_n})} \le C_n,
 \end{equation}
 where the norm $ \Vert w \Vert_{C^{\alpha,\alpha/2}(E)}$ for a function $w=w(z,s)$ on $E \subset \mathbb R^{N+1}$  is given by 
 \begin{equation}
 \label{Holder norm definition}
 \Vert w \Vert_{C^{\alpha,\alpha/2}(E)} = \sup_E |w| +\sup_{\substack{(z,s), (\tilde{z},\tilde{s}) \in E\\ (z,s)\not= (\tilde{z},\tilde{s})}}\frac {|w(z,s)-w(\tilde{z},\tilde{s})|}{|z-\tilde{z}|^\alpha+|s-\tilde{s}|^{\alpha/2}}.
 \end{equation}

 \subsection{Getting the first term in the asymptotic formula}
 \label{subsection the first term}
  By virtue of \eqref{interior estimates}, it follows from the Arzel\`a-Ascoli theorem together with the Cantor diagonal process that there exist a sequence $\{\varepsilon_j\}$ with $\lim\limits_{j\to \infty}\varepsilon_j=0$ and a function $v^*$ on $\mathbb R^N\times(0,+\infty)$ which satisfy the following for every $n \in \mathbb N$:
  \begin{eqnarray}
  \{v^{\varepsilon_j}\}_{j\ge n} \mbox{ converges to } v^* \mbox{ as } j \to \infty \mbox{ uniformly on } \overline{Q_n};\ \ &&\label{uniform convergence subsequence 1st}\\
   \{\nabla v^{\varepsilon_j}\}_{j\ge n} \mbox{ converges to } \nabla v^* \mbox{ as } j \to \infty \mbox{ uniformly on } \overline{Q_{n}};\ \ &&\label{uniform convergence subsequence for gradient 1st}
  \\
  \Vert v^*\!\Vert_{C^{\alpha,\alpha/2}(\overline{Q_n})} + \Vert\nabla v^*\!\Vert_{C^{\alpha,\alpha/2}(\overline{Q_n})} \le C_n,\ \ &&\label{estimate for limit function 1st}.
  \end{eqnarray}
   where  $C_n, \alpha$ are the constants given in \eqref{interior estimates}. Moreover, by \eqref{weak solution after scaling}, $v^*$ satisfies
  \begin{equation}
  0 \le  v^*\le 1\ \mbox{ and } v_s^*=\mbox{div}(\phi^\prime(v^*) \nabla v^*)  \ \mbox{ in } \mathbb R^N\times (0,+\infty).\label{limit equation 1st} 
 \end{equation}
 Here  we used the fact that, as $\varepsilon\to 0^+$, $A^*$ converges to $A(0) = I_N$ uniformly on every compact set in $\mathbb R^N$. 
 
 Next lemma shows that $v^*(z,s)$ is explicitly given by $\Psi(z_N,s)$, where $\Psi$ is defined by \eqref{one-dimensional self-similar solution},  and it must be the zeroth-order approximation in $\varepsilon$ as $\varepsilon \to 0^+$ of $v^\varepsilon$.
 \begin{lemma}
 \label{the zeroth-order approximation}
 The following assertions hold:
 \begin{itemize}
 \item[\rm (i)] For each $n\in\mathbb N$, as $\varepsilon \to 0^+$, $|v^\varepsilon(z,s)- \Psi(z_N,s)|\le \varepsilon(\Lambda\sqrt{s} +o(1))$ for $(z,s)\in Q_n$;
 \item[\rm (ii)] $v^*(z,s) = \Psi(z_N,s)$ for all $(z,s) \in \mathbb R^N\times(0,+\infty)$;
 \item[\rm (iii)] For each $n\in\mathbb N$, $ \{v^{\varepsilon}\}_{\varepsilon\le (n+2)^{-1}}$ converges to $ v^*$ as  $\varepsilon \to 0^+$ uniformly on $\overline{Q_n}$;
  \item[\rm (iv)]For each $n\in\mathbb N$,  $ \{ \nabla v^{\varepsilon}\}_{\varepsilon\le (n+2)^{-1}}$  converges to $\nabla v^*$ as $ \varepsilon \to 0^+$ uniformly on $\overline{Q_{n}}$,
 \end{itemize}
 where $\Lambda$ is the constant given by Lemma \ref{A local key estimate from above and below} and $\Psi$ is the function given by \eqref{one-dimensional self-similar solution}.
 \end{lemma}
 
 \noindent
 {\it Proof. } (i) together with \eqref{uniform convergence subsequence 1st} yields (ii). Both (iii) ad (iv) follow from (ii) with the aid of the compactness arguments in the beginning of  subsection \ref{subsection the first term}. It remains to prove (i). Let $d^*=d^*(x)$ be the signed distance function to the boundary $\partial\Omega$.
  Let $n \in \mathbb N$ and $0 < \varepsilon \le (n+2)^{-1}$. Then $\overline{B_{n\rho_0}(0)}\subset B_{\varepsilon^{-1}\rho_0}(0)$. Let $d^*=d^*(x)$ be the signed distance function to the boundary $\partial\Omega$. Then $d^*$ satisfies
\begin{equation}
\label{Taylor expansions of the signed distance function}
d^*(x) = -x_N + \frac 12\sum_{j=1}^{N-1}\kappa_j x_j^2 + o(|x|^2) \mbox{ as } x \to 0.
\end{equation}
Since $\hat{x}=\varepsilon\hat{z}$ and $x_N=\varphi\!\left(\varepsilon\hat{z}\right)-\varepsilon z_N$, we observe that, for each $z \in B_{n\rho_0}(0)$,  
\begin{equation}
\label{asymptotics of xN}
x_N=\frac 12\sum_{j=1}^{N-1}\kappa_j\!\left(\varepsilon z_j\right)^2-\varepsilon z_N+o\!\left(\varepsilon^2\right)\ \mbox{ as } \varepsilon \to 0^+.
\end{equation}
Hence,  $|x|^2=\varepsilon^2|z|^2 + O\!\left(\varepsilon^3\right)  \mbox{ as } \varepsilon \to 0^+$, and combining  \eqref{Taylor expansions of the signed distance function} with \eqref{asymptotics of xN} yields that for each $z \in B_{n\rho_0}(0)$
\begin{equation}
\label{a key estimate}
d^*(x) = \varepsilon z_N + o\!\left(\varepsilon^2\right)\ \mbox{ as } \varepsilon \to 0^+.
\end{equation}
Let $(z,s)\in Q_n$. Since $t=\varepsilon^2s$, it follows from \eqref{a key estimate} that
\begin{equation}
\label{distance function with time}
t^{-\frac 12}d^*(x)=\varepsilon^{-1} s^{-\frac 12}(\varepsilon z_N + o\!\left(\varepsilon^2\right)) =s^{-\frac 12}z_N+o(\varepsilon)\ \mbox{ as } \varepsilon \to 0^+.
\end{equation}
Then, by Proposition \ref{bounds on f}  the function $f$ given by \eqref{for selfsimilar solutions} satisfies
\begin{equation}
\label{difference of f}
f(t^{-\frac 12}d^*(x))-f(s^{-\frac12}z_N)=o(\varepsilon)\ \mbox{ as }\varepsilon \to 0^+,
\end{equation}
 which implies that
 \begin{equation}
\label{difference of psi  and Psi}
\psi(x,t)-\Psi(z_N,s)=o(\varepsilon)\ \mbox{ as }\varepsilon \to 0^+,
\end{equation}
where $\psi, \Psi$ are given by \eqref{approximate solution near the boundary}, \eqref{one-dimensional self-similar solution}, respectively. Therefore, since $u(x,t)=v^\varepsilon(z,s)$, we have from Lemma \ref{A local key estimate from above and below} that as $\varepsilon \to 0^+$
$$
|v^\varepsilon(z,s)- \Psi(z_N,s)|\le|u(x,t)-\psi(x,t)|+|\psi(x,t)- \Psi(z_N,s)| \le \Lambda\sqrt{t} + o(\varepsilon) =\varepsilon \Lambda\sqrt{s} +o(\varepsilon),
$$
 which gives (i). \qed


 \subsection{Getting  the second term in the asymptotic formula}
 \label{subsection the second term}
Let us introduce the function $S^\varepsilon=S^\varepsilon(z, s)$ for $(z,s) \in B_{\varepsilon^{-1}\rho_0}(0)  \times (0,+\infty)$ by
\begin{equation}
\label{function for the second mean curvature term}
S^\varepsilon(z,s) = \varepsilon^{-1}(v^\varepsilon(z,s)-v^*(z,s)) \ \left\{ = \varepsilon^{-1}(v^\varepsilon(z,s)-\Psi(z_N,s)) \right\}.
\end{equation}
Then we have from (i) of Lemma \ref{the zeroth-order approximation} that for each $n\in\mathbb N$, as $\varepsilon \to 0^+$, 
\begin{equation}
\label{bounds of the second term}
|S^\varepsilon(z, s)|\le \Lambda\sqrt{s} +o(1)\ \mbox{ for } (z,s)\in Q_n.
\end{equation}

Let us derive the equation which $S^\varepsilon$ satisfies and the asymptotic behavior of $S^\varepsilon$ as $\varepsilon \to 0^+$. Since it follows from \eqref{function for the second mean curvature term} that
$$
v^\varepsilon(z,s) = \varepsilon S^\varepsilon(z,s) + \Psi(z_N,s),
$$
we see from \eqref{weak solution after scaling}  and \eqref{one-dimensional Cauchy} that in $B_{\varepsilon^{-1}\rho_0}(0)\times (0,+\infty)$
\begin{align}
S^\varepsilon_s &=\mbox{div}(\phi^\prime(v^\varepsilon)A^*\nabla S^\varepsilon) 
+ \mbox{div}\left(\int_0^1\phi^{\prime\prime}(\Psi+\theta\varepsilon S^\varepsilon)d\theta\,S^\varepsilon A^*\nabla\Psi\right)\nonumber\\
&\quad+ \mbox{div}(\varepsilon^{-1}\!\phi^\prime(\Psi)(A^*-I_N)\nabla \Psi).\label{weak solution to the second order approximate}
\end{align}
Let us compute in particular the third term of the right-hand side of \eqref{weak solution to the second order approximate}.  We observe that for $(z,s) \in B_{\varepsilon^{-1}\rho_0}(0)\times (0,+\infty)$ 
\begin{eqnarray*}
\nabla \Psi(z_N,s) &=& {}^t[0,\dots,0,\partial_{z_N}\Psi(z_N,s)],\\
A^*(z)-I_N&=&A\!\left(\varepsilon\hat{z}\right)-I_N= \left[\begin{array}{c|c} O&(\nabla \varphi)(\varepsilon\hat{z})\\ \hline
{}^t((\nabla\varphi)(\varepsilon\hat{z}))&|(\nabla\varphi)(\varepsilon\hat{z})|^2
\end{array}\right].
\end{eqnarray*}
Let $n \in \mathbb N, \ \varepsilon \le (n+2)^{-1}$ and $(z,s) \in Q_{n+1}$. Remark that $\overline{B_{(n+1)\rho_0}(0)}\subset B_{\varepsilon^{-1}\rho_0}(0)$.
Since  $\Vert\varphi\Vert_{C^2(\mathbb R^{N-1})}$ is finite and $\nabla\varphi(0)=0$,   for every $j=1,\dots,N\!-\!1$ and every $\varepsilon \le (n+2)^{-1}$, we have the following estimates:
\begin{eqnarray*}
&&\sup_{z \in \overline{B_{(n+1)\rho_0}(0)}}\!\!\varepsilon^{-1}|(\nabla\varphi)(\varepsilon\hat{z})| \le (n+1)\rho_0\, \Vert\varphi\Vert_{C^2(\mathbb R^{N-1})}\\
 &&\mbox{ and } \sup_{z \in \overline{B_{(n+1)\rho_0}(0)}}\!\!\varepsilon^{-1}|\partial_{z_j}\{(\nabla\varphi)(\varepsilon\hat{z})\}|\le \Vert\varphi\Vert_{C^2(\mathbb R^{N-1})}.
\end{eqnarray*}
Thus, we infer that there exists a constant $C_{1,n} >0$, which is independent of $q \in \partial\Omega\cap B_r(p)$, satisfying for every $\varepsilon\le(n+2)^{-1}$
\begin{equation}
\label{Lipschitz norm of coefficients for the equation}
\Vert \varepsilon^{-1}\, (A^*-I_N)\nabla \Psi\Vert_{C^{1,1/2}(\overline{Q_{n+1}})} \le C_{1,n},
\end{equation}
where the norm of the left-hand side means \eqref{Holder norm definition} with $\alpha=1$. Moreover, since $|\phi^{\prime\prime}|\le c$, it follows from \eqref{interior estimates} that there exists
a constant $C_{2,n} >0$, which is independent of $q \in \partial\Omega\cap \overline{B_r(p)}$, satisfying for every $\varepsilon\le(n+2)^{-1}$
\begin{equation}
\label{Holder norm of coefficients for the equation}
\Vert \phi^\prime(v^\varepsilon)A^*\Vert_{C^{\alpha,\alpha/2}(\overline{Q_{n+1}})} \le C_{2,n}\ \mbox {and }\  \Vert \varepsilon^{-1}\, \phi^\prime(v^\varepsilon)(A^*-I_N)\nabla \Psi\Vert_{C^{\alpha,\alpha/2}(\overline{Q_{n+1}})} \le C_{2,n}.
\end{equation}
These take care of the first and  third terms of the right-hand side of \eqref{weak solution to the second order approximate}.
For the second term of the right-hand side of \eqref{weak solution to the second order approximate}, we remark that 
$$
\Psi+\theta\varepsilon S^\varepsilon =\theta v^\varepsilon + (1-\theta)\Psi.
$$
Then, by virtue of \eqref{bounds of the second term} for $n+1$, \eqref{weak solution to the second order approximate}, \eqref{Holder norm of coefficients for the equation} and \eqref{interior estimates}, since  $|\phi^{\prime\prime}|\le c$ and  $|\phi^{\prime\prime\prime}|\le c$, we can again use the interior H\"older estimates \cite[Theorem 4.8, pp. 56--57]{Li1996} to have a constant $C_{3,n} > 0$ being independent of $q \in \partial\Omega\cap \overline{B_r(p)}$ such that for every $\varepsilon\le(n+2)^{-1}$
\begin{equation}
\label{interior estimates for S varepsilon}
 \Vert S^\varepsilon\Vert_{C^{\alpha,\alpha/2}(\overline{Q_n})} + \Vert\nabla S^\varepsilon\Vert_{C^{\alpha,\alpha/2}(\overline{Q_n})} \le C_{3,n},
 \end{equation}
 where  $\alpha\in (0,1)$ is the same number as in \eqref{interior estimates} which is independent of $n\in\mathbb N$ and $q \in \partial\Omega\cap \overline{B_r(p)}$.  Therefore, by virtue of \eqref{interior estimates for S varepsilon}, it follows from the Arzel\`a-Ascoli theorem together with the Cantor diagonal process that there exist a sequence $\{\varepsilon_j\}$ with $\lim\limits_{j\to \infty}\varepsilon_j=0$ and a function $S^*$ on $\mathbb R^N\times (0,+\infty)$ which satisfy the following for every $n \in \mathbb N$:
 \begin{eqnarray*}
  \{S^{\varepsilon_j}\}_{j\ge n} \mbox{ converges to } S^* \mbox{ as } j \to \infty \mbox{ uniformly on } \overline{Q_n}; &&\\
   \{\nabla S^{\varepsilon_j}\}_{j\ge n} \mbox{ converges to } \nabla S^* \mbox{ as } j \to \infty \mbox{ uniformly on } \overline{Q_{n}}; &&
  \\
  \Vert S^*\!\Vert_{C^{\alpha,\alpha/2}(\overline{Q_n})} + \Vert\nabla  S^*\!\Vert_{C^{\alpha,\alpha/2}(\overline{Q_n})} \le C_{3,n}, &&
  \end{eqnarray*}
   where  $C_{3,n}, \alpha$ are the same constants as in \eqref{interior estimates for S varepsilon}.   Moreover, since for $i=1,\dots,N\!-\!1$ 
$$
(\partial_{x_i} \varphi)(\varepsilon\hat{z}) = \sum_{k=1}^{N-1}(\partial_{x_k} \partial_{x_i} \varphi)(0)\varepsilon z_k + o(\varepsilon)
=\kappa_i\varepsilon z_i + o(\varepsilon)\ \mbox{ as } \varepsilon \to 0^+,
$$
$\varepsilon^{-1}\, \phi^\prime(v^\varepsilon)(A^*-I_N)\nabla \Psi$ converges to $\phi^\prime(\Psi)\partial_{z_N}\Psi\,{}^t[\kappa_1z_1,\dots,\kappa_{N-1}z_{N-1}, 0]$ as $\varepsilon \to 0^+$ uniformly  on every compact set in $\mathbb R^N\times (0,+\infty)$. 
Thus we infer from \eqref{bounds of the second term} and \eqref{weak solution to the second order approximate} that in $\mathbb R_N\times(0,+\infty)$
\begin{equation}\label{limit equation 2nd}
S^*_s =\mbox{div}(\phi^\prime(\Psi)\nabla S^*) + \partial_{z_N}(\phi^{\prime\prime}(\Psi)S^*\partial_{z_N}\Psi) +(N-1)H(0)\phi^\prime(\Psi)\partial_{z_N}\Psi,
\end{equation}
and 
  \begin{equation}\label{upper bound of S*}
  |S^*(z,s)| \le \Lambda\sqrt{s}\ \mbox{ for every } (z,s) \in \mathbb R^N\times(0,+\infty),
  \end{equation}
 where  we also used the fact that, as $\varepsilon\to 0^+$, $A^*$ converges to $A(0) = I_N$ uniformly on every compact set in $\mathbb R^N$. These guarantee that, for each $\tau > 0$, $S^*$ must be the unique bounded solution of the Cauchy problem  in $\mathbb R^N\times(0, \tau]$ for equation \eqref{limit equation 2nd} with initial condition $S^*=0$ (see \cite[Theorem 2, p. 639]{Ar1968} for the uniqueness). Furthermore, since $\Psi$ is independent of the variables $z_1,\dots, z_{N-1}$, we infer that $S^*$ is also independent of the variables $z_1,\dots, z_{N-1}$. Indeed,
 if we set $P(z,s) = S^*(z_1+a_1,\dots, z_{N-1}+a_{N-1}, z_N, s)$ for an arbitrary vector $(a_1,\dots. a_{N-1})\in\mathbb R^{N-1}$, then $P$ also satisfies the same Cauchy problem. Hence by the uniqueness $S^*\equiv P$. Thus we may write $S^*=S^*(z_N,s)$. 
 
 Let us distinguish the following two cases: 
 $$
 {\rm (I)}\  H(0)=0;\quad {\rm (II)}\ H(0)\not=0.
 $$
 In case (I), by the uniqueness of the solution of the Cauchy problem, $S^*\equiv 0$. In case (II), by setting $G =  \frac {S^*}{(N-1)H(0)}$ we regard $G$ as the unique bounded solution of the following Cauchy problem for every $\tau > 0$:
 \begin{align}
 G_s &=\mbox{div}(\phi^\prime(\Psi) \nabla G) +\partial_{z_N}(\phi^{\prime\prime}(\Psi)G\partial_{z_N}\Psi)+\phi^\prime(\Psi)\partial_{z_N}\Psi \ &\mbox{ in } \mathbb R^N\times(0, \tau],\\
  G&=0\ &\mbox{ on }  \mathbb R^N\times\{0\},
 \end{align}
 which is independent of the constant $(N-1)H(0)$. Since $\phi^\prime(\Psi)\partial_{z_N}\Psi > 0$, by the maximum principle
 \begin{equation}\label{positive G}
 G > 0 \ \mbox{ in }\ \mathbb R^N\times (0, +\infty).
 \end{equation}
  In particular, we observe that the positive value $G(0,1)$ depends only on $\phi$ and it is of course independent of $N$.
Moreover, observe that the function $g=s\partial_{z_N}\Psi$ satisfies
 \begin{align}
 g_s &=\mbox{div}(\phi^\prime(\Psi) \nabla g) +\partial_{z_N}(\phi^{\prime\prime}(\Psi)g\partial_{z_N}\Psi)+\partial_{z_N}\Psi \ &\mbox{ in } \mathbb R^N\times(0, \tau],\\
  g&=0\ &\mbox{ on }  \mathbb R^N\times\{0\}.
 \end{align}
 Hence, recalling that $\delta_1\le\phi^\prime\le\delta_2$ and utilizing the comparison principle yield that
 \begin{equation}\label{estimate of function G}
\delta_1s\,\partial_{z_N}\Psi\le G\le\delta_2s\, \partial_{z_N}\Psi\ \mbox{ for every } (z_N,s)\in \mathbb R\times (0, +\infty).
\end{equation}
Thus, the positive value $G(0,1)$ is estimated from above and below as
\begin{equation}\label{estimate of positive value G(0,1)}
\delta_1f^\prime(0) \le G(0,1) \le \delta_2f^\prime(0), 
\end{equation}
 and
 $$
 S^*(0,1)=G(0,1) (N-1)H(0).
 $$

 Once the limit function $S^*$ is uniquely determined, the compactness arguments again show that for every $n \in \mathbb N$:
 \begin{eqnarray}
  \{S^{\varepsilon}\}_{\varepsilon\le (n+2)^{-1}} \mbox{ converges to } S^* \mbox{ as } \varepsilon \to 0^+ \mbox{ uniformly on } \overline{Q_n}; &&\label{uniform convergence original sequence 2nd}\\
   \{\nabla S^{\varepsilon}\}_{\varepsilon\le (n+2)^{-1}} \mbox{ converges to } \nabla S^* \mbox{ as } \varepsilon\to 0^+ \mbox{ uniformly on } \overline{Q_{n}}.&&\label{uniform convergence original sequence 2nd gradiet}
\end{eqnarray}
Namely, $S^*$ is regarded as the first-order approximation in $\varepsilon$  as $\varepsilon \to 0^+$ of $v^\varepsilon$.
In particular, by setting $s=1, \varepsilon =\sqrt{t}$ and $z=0$, we observe that 
\begin{equation}
\label{at q on the boundary}
S^\varepsilon(0,1) = t^{-\frac12}\!\!\left(u(0,t)-f(0)\right) \mbox{ and }  S^*(0,1) =G(0,1)(N\!-\!1)H(0).
\end{equation}
Therefore, by setting $G(0,1) = \gamma(\phi)$, we get  the conclusion \eqref{asymptotic formula in time} of Theorem \ref{th: asymptotic formula in time} from \eqref{uniform convergence original sequence 2nd}.

 \subsection{The uniform convergence on $\partial\Omega\cap\overline{B_r(p)}$}
 \label{uniform convergence}
 This subsection is almost the same as \cite[section 4.6]{Sa2025}, but we add this section for the reader's convenience.
 Let us prove  that the convergence  in \eqref{uniform convergence original sequence 2nd} at $(0,1)\in \mathbb R^{N}\times(0,+\infty)$ is uniform on $q \in \partial\Omega\cap\overline{B_r(p)}$, which suffices to prove the last conclusion of  Theorem \ref{th: asymptotic formula in time}. Suppose that this is not the case. Then there exist a number $\eta > 0$, a sequence $\{\varepsilon_m\}$ with $\lim\limits_{m \to \infty}\varepsilon_m=0$, and a sequence $\{ q_m\} \subset \partial\Omega\cap\overline{B_r(p)}$ such that
 \begin{equation}
 \label{the contrary}
 |S^{\varepsilon_m}_{q_m}(0,1) - S^*_{q_m}(0,1)| \ge \eta\ \mbox{ for every } m \in \mathbb N,
 \end{equation}
where $S^\varepsilon_{q_m}, S^*_{q_m}$ mean $S^{\varepsilon}, S^*$, whose $q \in \partial\Omega\cap\overline{B_r(p)}$ is replaced with $q_m\in\partial\Omega\cap\overline{B_r(p)}$, respectively. Recall that at their principal coordinate systems for $q_m\ (m \in \mathbb N)$, the origin $0$ corresponds to each $q_m$. Since $\partial\Omega\cap\overline{B_r(p)}$ is compact, it follows from the Bolzano-Weierstrass theorem that  $\{ q_m\}$ has a convergent subsequence. Thus, for brevity, we may assume that $\{ q_m\}$ itself converges to a point $q_*\in\partial\Omega\cap\overline{B_r(p)}$ as $m\to \infty$. Hence $H(q_m) \to H(q_*)$ as $m \to \infty$, since $\partial\Omega\cap B_\rho(p)$ is of class $C^2$. By the second equality of \eqref{at q on the boundary}, we infer that
\begin{equation}
\label{one step to get a contradiction}
 S^*_{q_m}(0,1) \to  S^*_{q_*}(0,1)=\gamma(\phi)(N-1)H(q_*) \ \mbox{ as } m \to \infty,
\end{equation}
where $S^*_{q_*}$ means $S^*$ whose  $q \in \partial\Omega\cap\overline{B_r(p)}$ is replaced with $q_*\in\partial\Omega\cap\overline{B_r(p)}$.
Observe from \eqref{interior estimates for S varepsilon} that $\{ S^{\varepsilon_m}_{q_m}\}$ may also satisfy the interior estimates: for every $n, m \in \mathbb N$
\begin{equation*}
 \Vert S^{\varepsilon_m}_{q_m}\Vert_{C^{\alpha,\alpha/2}(\overline{Q_n})} + \Vert\nabla S^{\varepsilon_m}_{q_m}\Vert_{C^{\alpha,\alpha/2}(\overline{Q_n})} \le C_{3,n}.\qquad
 \end{equation*}
Then it follows from the Arzel\`a-Ascoli theorem together with the Cantor diagonal process  that there exists a convergent subsequence of $\{ S^{\varepsilon_m}_{q_m}\}$. For brevity, let $\{ S^{\varepsilon_m}_{q_m}\}$ itself converge. In particular, since $q_m \to q_*$ as $m \to \infty$, as in the previous blow-up arguments, we may infer that as $m \to \infty$
 \begin{equation}
 \label{convergence to S_* at q_*}
 \{ S^{\varepsilon_m}_{q_m}\}\mbox{ converges to } S^*_{q_*}\ \mbox{ uniformly on every compact set in }\mathbb R^N\times(0,+\infty).
 \end{equation}
  Hence it follows from \eqref{convergence to S_* at q_*}  that
 \begin{equation}
 \label{the other step to get a contradiction}
 S^{\varepsilon_m}_{q_m}(0,1) \to S^*_{q_*}(0,1)=\gamma(\phi)(N-1)H(q_*) \ \mbox{ as } m \to \infty.
\end{equation}
This together with \eqref{one step to get a contradiction} yields a contradiction to \eqref{the contrary}.

\setcounter{equation}{0}
\setcounter{theorem}{0}

\section{Proofs of Theorems \ref{th:hyperplane} and \ref{th:sphere theorem}}
\label{section_Applications}

Let us prove Theorems  \ref{th:hyperplane} and \ref{th:sphere theorem}.

\noindent
{\it Proof of Theorem \ref{th:hyperplane}.}\ Although the proof is the same as in \cite[Proof of Theorem 1.3]{Sa2025}, for the reader's convenience, we give it because it is short. By Corollary \ref{constant mean curvature}, the mean curvature $H$ of $\partial\Omega$ must be constant. Since $\partial \Omega$ is a entire graph, by the divergence theorem $H$ must vanish.
Hence $\partial\Omega$ is an entire minimal graph over $\mathbb R^{N-1}$. If $N=2$, the mean curvature is simply a curvature, and hence $\partial\Omega$ must be a straight line. If \  $3\le N\le 8$, by the Bernstein theorem for the minimal surface equation (see \cite[Theorem 17.8, p.208]{G1984}), $\partial\Omega$ must be a hyperplane, and   if $\nabla \varphi$ is bounded with $ N \ge 3$, by Moser's theorem\cite[Corollary, p.591]{M1961} (see also \cite[Theorem 17.5, p.205]{G1984}), $\partial\Omega$ must be a hyperplane. \qed

\vskip 2ex

\noindent
{\it Proofs of Theorem \ref{th:sphere theorem}.}\ By Corollary \ref{constant mean curvature} together with Alexandrov's sphere theorem \cite[p.412]{Al1958}, $\Gamma$ must be a sphere. Since either the inside of $\Gamma$ or the outside of $\Gamma$ is included in one of the two sets, $\Omega$ and $\mathbb R^N\setminus\overline{\Omega}$, for instance let us consider the case where the inside $D$ of $\Gamma$ is included  in $\Omega$. Then, by taking into account  the initial data $\mathcal X_\Omega$ and the overdetermined condition \eqref{stationary level surface with one C2 bounded part}, we infer from the uniqueness of the solution of the initial-boundary value problem for the nonlinear diffusion equation $u_t=\Delta \phi(u)$ that, for every $t >0$, $u$ is radially symmetric in $x\in D$ with respect to the center of $\Gamma$. With the aid of translation, we may assume that the center of $\Gamma$ is at the origin. Let $M$ be a $N\times N$ orthogonal matrix, and define the function $v=v(x,t)$ by
\begin{equation}\label{by an orthogonal matrix M}
v(x,t)=u(Mx,t)\ \mbox{ for } (x,t)\in\mathbb R^N\times [0, +\infty).
\end{equation}
Then, $v$ also satisfies the same nonlinear diffusion equation $v_t=\Delta\phi(v)$ in $ \mathbb R^N\times (0, +\infty)$. Set $w=u-v$. Thus, since 
$$
\phi^\prime(u)-\phi^\prime(v) = \int_0^1\phi^{\prime\prime}(v+\theta(u-v))d\theta\, w,
$$ 
by setting $b=b(x,t)= \int_0^1\phi^{\prime\prime}(v+\theta(u-v))d\theta$ we have
\begin{equation}\label{our equation for w}
w_t=\mbox{ div}(\phi^\prime\nabla w) + b\nabla v\cdot\nabla w+\mbox{div}(b\nabla v)w\ \mbox{ in } \mathbb R^N\times (0,+\infty).
\end{equation}
We know that $w\equiv 0$ in $D \times (0, +\infty)$,  and hence we can apply the unique continuation theorems for parabolic operators to $w$  to get 
$$
w\equiv 0\ \mbox{  in } \mathbb R^N\times (0, +\infty).
$$
See \cite[Theorem 1, p. 37]{EF2003} or its stronger versions \cite[Theorem 3, p. 1599]{F2003} and  \cite[Theorem 1, p. 500]{AV2003} for the unique continuation theorems for parabolic operators we used here, and see also \cite{V2008} for a survey on the unique continuation properties for parabolic equations.
Since the matrix $M$ can be chosen arbitrarily, for every $t >0$, $u$ is radially symmetric in $x\in\mathbb R^N$. 
Suppose that  there is another component $\Gamma^*$ of $\partial\Omega$. Since $\partial\Omega$ is of class $C^0$,  $\Gamma^*$ must be a sphere having the same center as that of $\Gamma$. Therefore $\Gamma^*$ is of class $C^2$ and hence by  Theorem \ref{th: asymptotic formula in time} with the overdetermined condition \eqref{stationary level surface with one C2 bounded part} the radius of $\Gamma^*$ must be the same as that of $\Gamma$. This is a contradiction. The other cases can be dealt with similarly.  \qed

\bigskip



\begin{thebibliography}{ACM}


\bibitem[AV]{AV2003} G.~Alessandrini and S.~Vessella, Remark of the strong unique continuation property for parabolic operators, Proc. Amer. Math. Soc., \textbf{132} (2003), 499--501.

\bibitem[Al] {Al1958} A.~D.~Alexandrov, Uniqueness theorems for surfaces in the large V, Vestnik Leningrad
Univ. \textbf{13} (19) (1958), 5--8, English translation: Amer. Math. Soc. Transl. Ser. 2, \textbf{21} (1962), 412--416.

\bibitem[ACM]{ACM2018} L.~Ambrosio, A.~Carlotto and A.~Massaccesi, Lectures on Elliptic Partial Differential Equations, Edizioni della Normale Pisa, 2018.

\bibitem[Ar1]{Ar1967} D.~G.~Aronson, Bounds for the fundamental solutions of a parabolic equation,  Bull. Amer. Math. Soc., \textbf{73} (1967), 890--896.

\bibitem[Ar2]{Ar1968} D.~G.~Aronson, Non-negative solutions of linear parabolic equations, Ann. Scuola Norm. Sup. Pisa, \textbf{22} (1968), 607--694.

\bibitem[AP]{AP1974}
F.~V.~Atkinson and L.~A.~Peletier, Similarity solutions of the nonlinear diffusion equation,
Arch. Rational Mech. Anal. \textbf{54} (1974), 373--392. 






\bibitem[D]{Dong2012} H.~Dong, Gradient estimates for parabolic and elliptic systems from linear Lamiates, Arch. Rational Mech Anal., 205 (2012), 119--149.

\bibitem[EF]{EF2003} L.~Escauriaza and F.~J.~Fern\'andez, Unique continuation for parabolic operators, Ark. Mat., \textbf{41} (2003), 35--60.



\bibitem[E1]{E1993} L.~C.~Evans, Convergence of an algorithm for mean curvature motion, Indiana Univ. Math. J., \textbf{42} (1993), 533--557.

\bibitem[E2]{E2010} L.~C.~Evans, Partial Differential Equations, 2nd edition, American Math. Soc., Providence, RI, 2010.


\bibitem[FS]{FS1986} E.~Fabes and D.~Stroock,  A new proof of Moser's parabolic Harnack inequality using the old
ideas of Nash, Arch. Rational Mech. Anal., \textbf{96} (1986), 327--338.

\bibitem[F]{F2003} F.~J.~Fern\'andez, Unique continuation for parabolic operators II, Comm. Partial Differential Equations \textbf{28} (2003), 1597--1604.



\bibitem[G]{G1984} E.~Giusti, Minimal Surfaces and Functions of Bounded Variations, Birkh\"auser, Boston, Basel, Stuttgart, 1984.









\bibitem[Li]{Li1996}G.~M.~Lieberman, Second Order Parabolic Differential Equations, World Scientific, Singapore, New Jersey, London, Hong Kong, 1996.

\bibitem[MS1]{MS2010} R.~Magnanini and S.~Sakaguchi, Nonlinear diffusion with a bounded
stationary level surface, Ann. Inst. Henri Poincar\'e - (C) Anal. Non Lin\'eaire \textbf{27} (2010), 937--952.


\bibitem[MS2]{MS2012} R.~Magnanini and S.~Sakaguchi, Interaction between nonlinear diffusion
and geometry of domain, J. Differential Equations \textbf{252} (2012), 236--257.


\bibitem[Mo]{M1961} J.~Moser, On Harnack's theorem for elliptic differential equations, Comm. Pure Appl. Math. \textbf{14} (1961), 577--591.



\bibitem[NT]{NT1993} W.~M.~Ni and I.~Takagi, Locating the peaks of least-energy solutions to  a semilinear Neumann problem, Duke Math. J., \textbf{70} (1993), 247--281.

\bibitem[Sa1]{Sa2013} S.~Sakaguchi, Stationary level surfaces and Liouville-type theorems characterizing hyperplanes, 
in `` Geometric Properties of Parabolic and Elliptic PDE's ", Springer INdAM Series Vol. \textbf{2} (2013), 269--282.






\bibitem[Sa2]{Sa2025} S.~Sakaguchi, Interaction between initial behavior of temperature and the mean curvature of the interface in two-phase heat conductors, J. Geom.  Anal.,  \textbf{35} (2025), 224.



\bibitem[V]{V2008} S.~Vessella, Quantitative estimates of unique continuation for parabolic equations, determination of unknown time-varying boundaries and optimal stability estimates, Inverse Problems 
\textbf{24} (2008), 023001.






\end{thebibliography}
\end{document}